\documentclass[12pt, a4paper]{article}
\usepackage[sumlimits, namelimits]{amsmath}
\usepackage{longtable,multicol,verbatim,graphics,graphicx,lscape,animate}
\usepackage{amsmath,amsfonts,amsthm,amssymb,paralist,mathrsfs}
\usepackage[margin=2.5cm]{geometry}
\usepackage[all]{xy}
\numberwithin{equation}{subsection}
\theoremstyle{plain}
\newtheorem{thm}{THEOREM}[subsection]
\newtheorem{definition}[thm]{DEFINITION}
\newtheorem{corol}[thm]{COROLLARY}
\newtheorem{prove}{Proof of Theorem}[subsection]
\newtheorem{lemma}[thm]{LEMMA}
\newtheorem{prop}[thm]{PROPOSITION}

\theoremstyle{definition}

\newtheorem{rem}[thm]{REMARK}
\theoremstyle{break}
\newtheorem{lemma-1}[thm]{LEMMA}

\newcommand{\F}{\mathcal{F}}
\newcommand{\K}{\mathcal{K}}
\newcommand{\R}{\mathbb{R} \backslash \{0\}}

\usepackage[bookmarks,colorlinks,citecolor=black]{hyperref}

\title{\bf K-theory for the Leaf Space of Foliations Formed by
the Generic K-orbits of a class of Solvable Real Lie Groups{\footnote{This paper was supported by the University of Economics and Law, Vietnam National University-Ho Chi Minh City and the University of Education at Ho Chi Minh City.}}}
\author{{\bf Le Anh} ${\bf Vu}^1$,\, {\bf Duong Quang} ${\bf Hoa}^2$
and {\bf Nguyen Anh} ${\bf Tuan}^3$\\
${}^1$\footnotesize{Department of Mathematics and Economic Statistics, University of Economics and Law}\\
\footnotesize{Vietnam National University - Ho Chi Minh City, Viet nam}\\
\footnotesize{E-mail:\, vula@uel.edu.vn}\\
${}^2$ \footnotesize{E-mail:\, duongquanghoabt@yahoo.com.vn}\\
${}^3$ \footnotesize{E-mail:\, natuan@upes.edu.vn}}

\date{}

\DeclareMathOperator{\Ext}{Ext}
\DeclareMathOperator{\Hom}{Hom}
\begin{document}
\maketitle
\begin{abstract}
The paper is a continuation of the works \cite{VU-SH} of Vu and Shum, \cite{VU-HO09} and \cite{VU-HO10} of Vu and Hoa. In \cite{VU-SH}, Vu and Shum classified all the MD5-algebras having commutative derived ideals. In \cite{VU-HO09}, Vu and Hoa considered foliations formed by the maximal dimensional K-orbits (for short, MD5-foliations) of connected MD5-groups such that their Lie algebras have 4-dimensional commutative derived ideals and gave a topological classification of the considered foliations. In \cite{VU-HO10}, Vu and Hoa characterized the Connes' C*-algebras of some MD5-foliations considered in \cite{VU-HO09} by the method of K-functors. In this paper, we study K-theory for the leaf space of all MD5-foliations which are classified in \cite{VU-HO09} and characterize the Connes' C*-algebras of them by the method of K-functors.
\end{abstract}



\section*{INTRODUCTION}

\hskip6mm In the decades 1970s -- 1980s, works of D.N. Diep \cite{DI75}, J. Rosenberg \cite{RO82}, G. G. Kasparov \cite{KAS}, V. M. Son and H. H. Viet \cite{SO-VI},\ldots have shown that K-functors are well adapted to characterize a large class of group C*-algebras. Kirillov's method of orbits allows to find out the class of Lie groups MD, for which the group C*-algebras can be characterized by means of suitable K-functors. In terms of D. N. Diep, an MD-group of dimension n (for short, an MDn-group) is an n-dimensional solvable real Lie group whose orbits in the co-adjoint representation (i.e., the K- representation) are the orbits of zero or maximal dimension. The Lie algebra of an MDn-group is called an MDn-algebra (see \cite{DI99}).

In 1982, studying foliated manifolds, A. Connes \cite{CO82} introduced the notion of C*-algebra $C^{*}(V, \F)$ associated to a measured foliation $(V, \F)$. In general, the space of leaves $V/\F$ of a foliation $(V, \F)$ with the quotient topology is a fairly untractable topological space. The Connes' C*-algebra $C^{*}(V, \F)$ represents the leaf space $V/\F$ \,of\, $(V, \F)$ in the following sense: When the foliation $(V, \F)$ comes from a fibration (with connected fibers) $p: V \rightarrow B$ on some locally compact basis $B$, in particular the leaf space of $(V, \F)$ in this case is the basis $B$, then $C^{*}(V, \F)$ is isomorphic to ${C_0}(B)\otimes \K$, where $\K = \K({\mathcal{H}})$ denotes the C*-algebra of compact operators on an (infinite dimensional) separable Hilbert space ${\mathcal{H}}$ (see \cite{TO}). For such foliations $(V, \F)$, $K_{*}{\left( C^*(V, \F)\right)}$ coincides with the K-theory of the leaf space $B = V/\F$. Therefore, for an arbitrary foliation $(V, \F)$, A. Connes \cite{CO82} used $K_{*}{\left( C^*(V, \F)\right)}$ to define the K-theory for the leaf space of $(V, \F)$. 

The following question naturally arises: Can we characterize Connes' C*-algebras $C^*(V, \F)$ by using K-functors? In fact, A. M. Torpe \cite{TO} has shown that the method of K-functors is very useful and effective to characterize the structure of Connes C*-algebras associated with the foliations by Reeb components.

Now, we will deal with the class of MD-groups and MD-algebras. 
One of the reasons for studying the class MD is based on the following
fact: for every MD-group $G$, the family of all K-orbits having maximal
dimension forms a measured foliation in terms of Connes \cite{CO82}. This foliation is called MD-foliation associated with $G$. The C*-algebra
of $G$ can be easily characterized when the Connes' C*-algebra
of MD-foliation associated with $G$ is known. Hence, the problem of classifying the topology and describing Connes' C*-algebras of the class of MD-foliations is worth to study. Combining methods of Kirillov (see \cite{KI}) and Connes (see \cite{CO82}), Vu had studied MD4-foliations associated with all indecomposable connected MD4-groups and characterized Connes' C*-algebras of these foliations (see \cite{VU90-2}). Recently, Vu and Shum \cite{VU-SH} have classified, up to isomorphism, all the MD5-algebras having commutative derived ideals. In this classification, there are 14 families indecomposable MD5-algebras having 4-dimensional commutative derived ideals. So we obtain 14 families corresponding MD5-foliations.

In 2009, Vu and Hoa \cite{VU-HO09} gave a topological classification of 14 families of the corresponding MD5-foliations. Though there are 14 families considered MD5-foliations, there are only exactly 3 types of foliated topologies which are denoted by $\F_1, \,\F_2, \,\F_3$. All MD5-foliations of the type $\F_1$ are trivial fibrations with connected fibres on 3-dimesional sphere $S^3$ (see \cite{VU-HO09}), so Connes' C*-algebras of them are isomorphic to the C*-algebra $C(S^3) \otimes \K$ following \cite{CO82}. 

All MD5-foliations of the types $\F_2, \,\F_3$ are non-trivial and cannot come from fibrations. But these foliations can be given by suitable actions of ${\mathbb{R}}^2$ on the foliated manifolds (see \cite{VU-HO09}). In \cite{VU-HO10}, Vu and Hoa characterized Connes' C*-algebras associated to MD5-foliations of the type $\F_2$ by KK-functors. 

The purpose of this paper is to study K-theory for the leaf space and to characterize the structure of Connes' C*-algebras $C^*(V,\F)$ of all MD5-foliations $(V,\F)$  of the remaining type $\F_3$ by the method of K-functors. However, for the sake of completeness, we will consider concurrently both these types $\F_2, \,\F_3$ in this paper.

The paper is organized as follows: The first section deals with some preliminary notions and results. In particular, we recall in this section the main results from \cite{VU-HO09} and \cite{VU-HO10} on the classification of all MD5-algebras having 4-dimensional commutative derived ideals and the topological classifications of MD5-foliations associated to all of indecomposable connected MD5-groups which correspond with classified Lie algebras. Section 2 is devoted to study K-theory for the leaf space of all MD5-foliations of both these non-trivial types $\F_2, \,\F_3$. The main results of the paper are Theorems 2.1.1 and 2.2.3 in Section 2.

\section{PRELIMINARIES}

\hskip6mm In this section, we recall some notations and preliminary results which will be used later. Firstly, we recall some properties of Connes' C*-algebras in the first subsection. 

\subsection{Some Properties of Connes' C*-algebras}

\hskip6mm For a given foliation $(V, \F)$, the C*-algebra $C^*(V, \F)$ is generated by certain functions on the graph (or holonomy groupoid) $G$ (see \cite[Section 2]{TO}) of the foliations $(V, \F)$. The construction of $C^*(V, \F)$ is fairly complicated, so we do not recall this construction here and refer the reader to the work \cite[Sects. 5-6]{CO82} of A. Connes.
\begin{prop}[see {\bf \cite[Proposition 2.1.4]{TO}}]
Topologically equivalent foliations yield isomorphic C*-algebras. \hfill{$\square$}
\end{prop}
\begin{prop}[see {\bf \cite[Proposition 2.1.5]{TO}}] Assume that the foliation $(V, \F)$ comes from an action of a Lie group $H$ in such a way that the graph G is given as $V \times H$. Then $C^*(V, \F)$ is isomorphic to the reduced crossed product $C_0(V)\rtimes H$ of $C_0(V)$ by $H$. \hfill{$\square$}
\end{prop}
\begin{prop}[see {\bf \cite[Proposition 2.1.6]{TO}}] Assume that the foliation $(V, \F)$ is given by a fibration (with connected fibers) $p: V \rightarrow B$. Then the graph 
$G = \lbrace (x, y) \in V \times V \,\vert \,p(x) = p(y) \rbrace $ which is a submanifold of $V \times V$ and $C^*(V, \F)$ is isomorphic to 
$C_0(B)\otimes \K$. \hfill{$\square$}
\end{prop}
\begin{prop}[see {\bf \cite[Proposition 2.1.7]{TO}}] Let $(V, \F)$ be a foliation. If $V^{'}$ $(\subset V)$ is an open subset of foliated manifold V and ${\F}^{'}$ is the restriction of $\F$ to $V^{'}$. Then the graph $G^{'}$ of $(V^{'}, \,{\F}^{'})$ is an open subset in the graph $G$ of $(V, \F)$ and $C^*(V^{'}, {\F}^{'})$ can be canonically embedded into $C^*(V, \F)$. \hfill{$\square$}
\end{prop}
\begin{prop}[see {\bf \cite[Subsection 2.2]{TO}}] Let $V^{'}(\subset V)$ be a saturated open subset in the foliated manifold V of a foliation $(V, \F)$. Denote by \,${\F}^{'},\, {\F}^{''}$ the restrictions of $\F$ on $V^{'}$ and $V\setminus V^{'}$, respectively. Then $C^*(V^{'}, \,{\F}^{'})$ is an ideal in $C^*(V, \F)$. Moreover, if the foliation $(V, \F)$ is given by a suitable action of an amenable Lie group $H$ such that the subgraph $G\setminus G^{'}$ is given as $(V\setminus V^{'})\times H$ then $C^*(V, \F)$ can be canonically embedded into the following exact sequence
\begin{equation}
\xymatrix{0 \ar[r] & C^*(V^{'}, \,{\F}^{'}) \ar[r] & C^*(V,\F) \ar[r] & C^*(V\setminus V^{'}, \,{\F}^{''}) \ar[r] & 0}. 
\end{equation} \hfill{$\square$}
\end{prop}

In the next subsection, we present the main results from \cite{VU-SH}. That is the classifications of indecomposable MD5-algebras having 4-dimensional commutative derived ideals.

\subsection{The Classification of Indecomposable MD5-algebras having 4-dimensional commutative dereived ideals}

\hskip6mm Firstly, we recall the notion of MD-groups and MD-algebras.
\begin{definition}[see {\bf \cite[Chapter 4, Definition 1.1]{DI75}}]An n-dimensional MD-group (for short, MDn-group) is an n-dimensional real solvable Lie group such that its K-orbits are orbits of dimension zero or maximal dimension. The Lie algebra
of an MDn-group is called an n-dimensional MD-algebra (for short, MDn-algebra).
\end{definition}

Now we use $\mathcal{G}$ to denote a Lie algebra of
dimension 5. We always choose a suitable basis $(X_{1}, X_{2}, X_{3},$
\\ $X_{4}, X_{5})$ in $\mathcal{G}$ so that
$\mathcal{G}$ is isomorphic to ${\mathbb{R}}^{5}$ as a real vector
space. The following proposition gives the classification of all indecomposable MD5-algebras having 4-dimensional commutative ideals.

\begin{prop}[see {\bf \cite[Theorem 2.1]{VU-SH}}]
Let $\mathcal{G}$ be an indecomposable MD5-algebra having 4-dimensional commutative $\mathcal{G}^{1}:$= $[\mathcal{G}, \mathcal{G}]
=\mathbb{R}.X_{2} \oplus \mathbb{R}.X_{3} \oplus
\mathbb{R}.X_{4} \oplus \mathbb{R}.X_{5} \equiv {\mathbb{R}}^{4}$,\,
$ad_{X_{1}} \in End({\mathcal{G}}^{1}) \equiv
Mat_{4}(\mathbb{R})$. Then $\mathcal{G}$ is isomorphic
to one and only one of the following Lie algebras.\\
\begin{itemize}
    \item[1.]${\mathcal{G}}_{5,4,1({\lambda}_{1}, {\lambda}_{2}, {\lambda}_{3})}:$
 $ad_{{X}_1} = \begin{pmatrix} {{\lambda}_1}&0&0&0\\
 0&{{\lambda}_2}&0&0\\0&0&{\lambda}_{3}&0\\0&0&0&1\end{pmatrix};\,
  {\lambda}_1, {\lambda}_2, {\lambda}_3 \in \mathbb{R}\setminus
  \lbrace 0, 1\rbrace,\, {\lambda}_1 \neq {\lambda}_2 \neq {\lambda}_3 \neq
  {\lambda}_1.$

    \item[2.]${\mathcal{G}}_{5,4,2({\lambda}_{1}, {\lambda}_{2})}:$
$ad_{{X}_1} = \begin{pmatrix} {{\lambda}_1}&0&0&0\\
0&{{\lambda}_2}&0&0\\0&0&1&0\\0&0&0&1\end{pmatrix};\,
{\lambda}_{1}, {\lambda}_{2} \in \mathbb{R}\setminus \lbrace 0, 1
\rbrace , {\lambda}_1 \neq {\lambda}_2. $ 

    \item[3.]${\mathcal{G}}_{5,4,3(\lambda)}:$
 $ad_{{X}_1} = \begin{pmatrix}
 {\lambda}&0&0&0\\0&{\lambda}&0&0\\0&0&1&0\\0&0&0&1 \end{pmatrix}; \,
 {\lambda} \in \mathbb{R}\setminus \lbrace 0, 1 \rbrace .$

    \item[4.]${\mathcal{G}}_{5,4,4(\lambda)}:$
$ad_{{X}_1} = \begin{pmatrix} {\lambda}&0&0&0\\0&1&0&0\\
0&0&1&0\\0&0&0&1 \end{pmatrix};\quad {\lambda} \in
\mathbb{R}\setminus \lbrace 0, 1 \rbrace.$

    \item[5.]${\mathcal{G}}_{5,4,5}:$
$ad_{{X}_1} = \begin{pmatrix} 1&0&0&0\\0&1&0&0\\
0&0&1&0\\0&0&0&1 \end{pmatrix}.$

    \item[6.]${\mathcal{G}}_{5,4,6({\lambda}_{1}, {\lambda}_{2})}$ :
$ad_{{X}_1} = \begin{pmatrix} {{\lambda}_1}&0&0&0\\
0&{{\lambda}_2}&0&0\\0&0&1&1\\0&0&0&1\end{pmatrix};\,
{\lambda}_{1}, {\lambda}_{2} \in \mathbb{R}\setminus \lbrace 0, 1
\rbrace , {\lambda}_1 \neq {\lambda}_2.$

    \item[7.]${\mathcal{G}}_{5,4,7(\lambda)}:$
$ad_{{X}_1} = \begin{pmatrix}
{\lambda}&0&0&0\\0&{\lambda}&0&0\\0&0&1&1\\0&0&0&1 \end{pmatrix};
\, {\lambda} \in \mathbb{R}\setminus \lbrace 0, 1 \rbrace .$

    \item[8.]${\mathcal{G}}_{5,4,8(\lambda)}:$
$ad_{{X}_1} = \begin{pmatrix}
{\lambda}&1&0&0\\0&{\lambda}&0&0\\0&0&1&1\\0&0&0&1 \end{pmatrix};
\, {\lambda} \in \mathbb{R}\setminus \lbrace 0, 1 \rbrace .$

    \item[9.]${\mathcal{G}}_{5,4,9(\lambda)}:$
$ad_{{X}_1} = \begin{pmatrix}
{\lambda}&0&0&0\\0&1&1&0\\0&0&1&1\\0&0&0&1 \end{pmatrix}; \,
{\lambda} \in \mathbb{R}\setminus \lbrace 0, 1\rbrace .$

    \item[10.]${\mathcal{G}}_{5,4,10}:$
$ad_{{X}_1} = \begin{pmatrix} 1&1&0&0\\0&1&1&0\\
0&0&1&1\\0&0&0&1 \end{pmatrix}.$

    \item[11.]${\mathcal{G}}_{5,4,11({\lambda}_{1}, {\lambda}_{2},\varphi)}:$
$ad_{{X}_1} = \begin{pmatrix} \cos\varphi&-\sin\varphi&0&0\\
\sin\varphi&\cos\varphi&0&0\\0&0&{\lambda}_{1}&0\\0&0&0&{\lambda}_{2}\end{pmatrix};$
\,${\lambda}_{1}, {\lambda}_{2} \in \mathbb{R}\setminus \lbrace 0 \rbrace ,
{\lambda}_1 \neq {\lambda}_2,\varphi \in (0,\pi).$

    \item[12.]${\mathcal{G}}_{5,4,12(\lambda, \varphi)}:$
$ad_{{X}_1} = \begin{pmatrix} \cos\varphi&-\sin\varphi&0&0\\
\sin\varphi&\cos\varphi&0&0\\0&0&\lambda&0\\0&0&0&\lambda\end{pmatrix};\,
\lambda \in \mathbb{R}\setminus \lbrace 0 \rbrace,\varphi \in (0,\pi).$

    \item[13.]${\mathcal{G}}_{5,4,13(\lambda, \varphi)}:$
$ad_{{X}_1} = \begin{pmatrix} \cos\varphi&-\sin\varphi&0&0\\
\sin\varphi&\cos\varphi&0&0\\0&0&\lambda&1\\0&0&0&\lambda\end{pmatrix};\,
\lambda \in \mathbb{R}\setminus \lbrace 0 \rbrace,\varphi \in (0,\pi).$

    \item[14.]${\mathcal{G}}_{5,4,14(\lambda, \mu, \varphi)}:$
$ad_{{X}_1} = \begin{pmatrix} \cos\varphi&-\sin\varphi&0&0\\
\sin\varphi&\cos\varphi&0&0\\0&0&\lambda&-\mu\\0&0&\mu&\lambda\end{pmatrix};$
\,$\lambda, \mu \in \mathbb{R}, \mu > 0, \varphi \in (0,\pi).$ \hfill{$\square$}
\end{itemize}
\end{prop}

We recall that each real Lie algebra $\mathcal{G}$ defines only one connected and simply connected Lie group $G$ such that $Lie(G)= \mathcal{G}$. Therefore, we obtain a collection of 14 families of connected and simply connected MD5-groups corresponding to the indecomposable MD5-algebras given in Proposition 1.2.2. For convenience, each MD5-group from this collection is also denoted by the same indices as corresponding MD5-algebra. For example, $G_{5,4,11({\lambda}_{1}, {\lambda}_{2},\varphi)}$ is the indecomposable connected and simply connected MD5-group corresponding to ${\mathcal{G}}_{5,4,11({\lambda}_{1}, {\lambda}_{2},\varphi)}$.

Before we study K-theory for the leaf spaces of MD5-foliations associated to  indecomposable MD5-groups corresponding to the MD5-algebras which are listed above, we need to recall the topological classifications of the class of these MD5-foliations in \cite{VU-HO09}.

\subsection{The Topological Classification of MD5-Foliations Associated to Considered MD5-groups}

\hskip6mm Let $G$ be one of all MD5-groups corresponding to the MD5-algebras listed in Proposition 1.2.2. Denote by ${\mathcal{G}}^{*}$ the dual space of the Lie algebra $\mathcal{G} = Lie(G)$ of $G$. Clearly,  ${\mathcal{G}}^{*}$ can be identified with ${\mathbb{R}}^{5}$ by fixing in it the basis $(\nobreak X_{1}^{*}, X_{2}^{*}, X_{3}^{*}, X_{4}^{*}, X_{5}^{*}\nobreak )$ which is the dual of the basis $(\nobreak X_{1}, X_{2}, X_{3}, X_{4}, X_{5} \nobreak)$ of $\mathcal{G}$. Let
$F = {\alpha}X_{1}^{*} + {\beta}X_{2}^{*} + {\gamma}X_{3}^{*} + {\delta}X_{4}^{*} + {\sigma}X_{5}^{*} \equiv ({\alpha}, {\beta}, {\gamma}, {\delta}, {\sigma})$ be an arbitrary element of ${\mathcal{G}}^{*} \equiv {\mathbb{R}}^{5}$. The notation ${\Omega}_{F}$ will be used to denote the K-orbit of $G$ which contains $F$. The geometrical picture of the K-orbits of $G$ is given by the following proposition which has been proved in \cite{VU-HO07}. 

\begin{prop}[see {\bf \cite[Theorems 3.3.1 -- 3.3.4]{VU-HO07}}] For each considered MD5-group $G$, the
K-orbit ${\Omega}_{F}$ of $G$ is described as follows.
\begin{description}
        \item [1.] Let $G$ be one of \,\, $G_{5,4,1({\lambda}_{1}, {\lambda}_{2},
{\lambda}_{3})}$, \,\,$G_{5,4,2({\lambda}_{1}, {\lambda}_{2})}$, \,\,$G_{5,4,3(\lambda)}$, \,\,$G_{5,4,4(\lambda)}$, \,\,$G_{5,4,5)}$, \,\,$G_{5,4,6({\lambda}_{1}, {\lambda}_{2})}$, \,\,$G_{5,4,7(\lambda)}$, \,\, $G_{5,4,8(\lambda)}$, \,\,$G_{5,4,9(\lambda)}$, \,\,$G_{5,4,10}$;
${\lambda}_{1}$, ${\lambda}_{2}$, ${\lambda}_{3},\lambda\in
\mathbb{R}\backslash\{0,1\}, \, {\lambda}_1 \neq {\lambda}_2 \neq {\lambda}_3 \neq
  {\lambda}_1$.
        \begin{description}
            \item[1.1.] If $\beta=\gamma=\delta=\sigma=0$ then $\Omega_F=\{F\}$ (the
0-dimensional orbit).
            \item[1.2.] If $\beta^2+\gamma^2+\delta^2+\sigma^2\neq0$ then $\Omega_F$
is the orbit of dimension 2 and it is one of the following:
        \end{description}
\begin{itemize}
    \item $\{(x,\beta{e^{a\lambda_1}},\gamma{e^{a\lambda_2}},\delta{e^{a\lambda_3}},
    \sigma{e^a}),x,a\in\mathbb{R}\}$ when $G=G_{5,4,1({\lambda}_{1},
    {\lambda}_{2}, {\lambda}_{3})}$.
    \item $\{(x,\beta{e^{a\lambda_1}},\gamma{e^{a\lambda_2}},\delta{e^{a}},
    \sigma{e^a}),x,a\in\mathbb{R}\}$ when $G=G_{5,4,2({\lambda}_{1},
    {\lambda}_{2})}$.
    \item $\{(x,\beta{e^{a\lambda}},\gamma{e^{a\lambda}},\delta{e^a},
    \sigma{e^a}),x,a\in\mathbb{R}\}$ when $G=G_{5,4,3(\lambda)}$.
    \item $\{(x,\beta{e^{a\lambda}},\gamma{e^a},\delta{e^{a}},
    \sigma{e^a}),x,a\in\mathbb{R}\}$ when $G=G_{5,4,4(\lambda)}$.
    \item $\{(x,\beta{e^a},\gamma{e^a},\delta{e^a},
    \sigma{e^a}),x,a\in\mathbb{R}\}$ when $G=G_{5,4,5}$.
    \item $\{(x,\beta{e^{a\lambda_1}},\gamma{e^{a\lambda_2}},\delta{e^{a}},\delta{a}{e^{a}}+
    \sigma{e^a}),x,a\in\mathbb{R}\}$ when $G=G_{5,4,6({\lambda}_{1},{\lambda}_{2})}$.
    \item $\{(x,\beta{e^{a\lambda}},\gamma{e^{a\lambda}},\delta{e^a},\delta{a}{e^{a}}+
    \sigma{e^a}),x,a\in\mathbb{R}\}$ when $G=G_{5,4,7(\lambda)}$.
    \item $\{(x,\beta{e^{a\lambda}},\beta{a}{e^{a\lambda}}+\gamma{e^{a\lambda}},
    \delta{e^a},\delta{a}{e^{a}}+\sigma{e^a}),x,a\in\mathbb{R}\}$ 
    when $G=G_{5,4,8(\lambda)}$.
    \item $\{(x,\beta{e^{a\lambda}},\gamma{e^a},\gamma{a}{e^a}+\delta{e^{a}},
    \frac{\gamma{a^2}e^a}{2}+\delta{a}{e^{a}}+\sigma{e^a}),x,a\in\mathbb{R}\}$
    when $G=G_{5,4,9(\lambda)}$.
    \item $\{(x,\beta{e^a},\beta{a}{e^a}+\gamma{e^a},\frac{\beta{a^2}e^a}{2}+
    \gamma{a}{e^a}+\delta{e^{a}},\frac{\beta{a^3e^a}}{6}+ \frac{\gamma{a^2}e^a}{2}+
    \delta{a}{e^{a}}+\sigma{e^a})$,

    \hfill {$x,a\in\mathbb{R}\}$} when $G=G_{5,4,10}$.
\end{itemize}
    \item [2.] Let G be one of \,\, $G_{5,4,11({\lambda}_{1}, {\lambda}_{2},\varphi)}$, \,\,$G_{5,4,12(\lambda,\varphi)}$, \,\,$G_{5,4,13(\lambda,\varphi)}$; $\lambda_1,\lambda_2,\lambda\in\mathbb{R}\setminus\{0\}$, ${\lambda}_1 \neq {\lambda}_2$, $\varphi\in(0,\pi)$. 
Let us identify \,\,${\mathcal{G}}_{5,4,11(\lambda_{1}, \lambda_{2},\varphi)}^*$, 
\,\,${\mathcal{G}}_{5,4,12(\lambda,\varphi)}^*$, \,\,${\mathcal{G}}_{5,4,13(\lambda,\varphi)}^*$ with $\mathbb{R}\times \mathbb{C}\times {\mathbb{R}}^2$ and $F$ with
$(\alpha,\beta+i\gamma,\delta,\sigma)$. Then we have
    \begin{description}
    \item[2.1.] If
$\beta+i\gamma=\delta=\sigma=0$ then $\Omega_F=\{F\}$ (the
0-dimensional orbit).
    \item[2.2.] If $|\beta+i\gamma|^2+\delta^2+\sigma^2\neq0$ then $\Omega_F$
is the orbit of dimension 2 and it is one of the following:
    \end{description}
\begin{itemize}
    \item $\{(x,(\beta+i\gamma)e^{ae^{-i\varphi}},\delta{e^{a\lambda_1}},
    \sigma{e^{a\lambda_2}}),x,a\in\mathbb{R}\}$
    when $G=G_{5,4,11({\lambda}_{1}, {\lambda}_{2},
\varphi)}$.
    \item $\{(x,(\beta+i\gamma)e^{ae^{-i\varphi}},\delta{e^{a\lambda}},
    \sigma{e^{a\lambda}}),x,a\in\mathbb{R}\}$
    when $G=G_{5,4,12(\lambda,\varphi)}$.
    \item $\{(x,(\beta+i\gamma)e^{ae^{-i\varphi}},\delta{e^{a\lambda}},
    \delta{a}{e^{a\lambda}}+\sigma{e^{a\lambda}}),x,a\in\mathbb{R}\}$ when
    $G=G_{5,4,13(\lambda,\varphi)}$.
\end{itemize}
    \item [3.] Let $G$ be
$G_{5,4,14(\lambda, \mu, \varphi)}$. Let us identify \,\,
${\mathcal{G}}_{5,4,14(\lambda,\mu,\varphi)}^*$ with
$\mathbb{R}\times \mathbb{C} \times \mathbb{C}$ and $F$ with
$(\alpha,\beta+i\gamma,\delta+i\sigma)$;
$\lambda,\mu\in\mathbb{R}$, \,\,$\mu>0$, \,\,$\varphi\in(o,\pi)$. Then we have
\begin{description}
    \item[3.1.] If $\beta+i\gamma=\delta+i\sigma=0$ then $\Omega_F=\{F\}$ (the
0-dimensional orbit).
    \item[3.2.] If $|\beta+i\gamma|^2+|\delta+i\sigma|^2\neq0$ then

$$\Omega_F=\{(x,(\beta+i\gamma)e^{ae^{-i\varphi}},(\delta+i\sigma)e^{a(\lambda-i\mu)}),x,a\in\mathbb{R}\}$$
(the 2-dimensional orbit). \hfill{$\square$}
\end{description}
\end{description} 
\end{prop}

In the introduction we have emphasized that, for every connected and
simply connected MD-group, the family of maximal-dimensional
K-orbits forms a measured foliation in terms of A. Connes.
Namely, we have the following proposition.

\begin{prop}[see {\bf \cite[Theorem 3.1]{VU-HO09}}] 
Let $G$ be one of the connected and simply connected MD5-groups corresponding to the MD5-algebras listed in Proposition 1.2.2, $\mathcal{F}_{G}$ be the family of all its K-orbits of dimension two and $V_{G}: = \bigcup \{ \Omega \,|\, \Omega \in \mathcal{F}_{G}\}$. Then $(V_{G},\,\mathcal{F}_{G})$ is a measurable foliation in the sense of Connes. We call it MD5-foliation associated with MD5-group
$G$. \hfill {$\square$}
\end{prop}

\subsection*{Remarks and Notations}

\hskip6mm Note that $V_{G}$ is an open submanifold of the dual space
$\mathcal{G}^{*} \equiv \mathbb{R}^{5}$ of the Lie algebra
$\mathcal{G}$ corresponding to $G$. Furthermore, for all MD5-groups of
the forms $G_{5,4, ...}$, the manifolds $V_{G}$ are diffeomorphic to
each other. So, for simplicity of notation, we shall write $(V,\,
F_{4, ...})$ instead of $(V_{G_{4, ...}},\, F_{G_{4, ...}})$. Now we recall the
classification of considered MD5-foliations (see \cite[Theorem 3.2]{VU-HO09}) in the following proposition below.

\begin{prop}[]{\bf (The Classification of Considered MD5-foliations)}
\begin{enumerate}
    \item [1.] There exist exactly 3 topological types of 14
    families of considered MD5-foliations as follows:
    \begin{enumerate}
        \item[1.1.]
        $\Bigl \{(V,{\mathcal{F}}_{4,1(\lambda_1,\lambda_2,\lambda_3)}),
        (V,{\mathcal{F}}_{4,2(\lambda_1,\lambda_2)}),
        (V,{\mathcal{F}}_{4,3(\lambda)}), (V,{\mathcal{F}}_{4,4(\lambda)}), 
        (V,{\mathcal{F}}_{4,5}), (V,{\mathcal{F}}_{4,6(\lambda_1,\lambda_2)}),$
        
        \hskip3cm $(V,{\mathcal{F}}_{4,7(\lambda)}), (V,{\mathcal{F}}_{4,8(\lambda)}), 
        (V,{\mathcal{F}}_{4,9(\lambda)})$, $(V,{\mathcal{F}}_{4,10}); 
        \lambda,\lambda_1, \lambda_2, \lambda_3 \in \mathbb{R}
        \backslash\{0,1\} \Bigl \}.$
        \item[1.2.]
        $\Bigl \{(V,{\mathcal{F}}_{4,11(\lambda_1,\lambda_2,\varphi)})$, 
        $(V,{\mathcal{F}}_{4,12(\lambda,\varphi)})$, 
        $(V,{\mathcal{F}}_{4,13(\lambda,\varphi)})$; 
        $\lambda,\,\,\, \lambda_1, \lambda_2 \in \mathbb{R}
        \backslash\{0\}$; $\varphi\in(0,\pi)\Bigl \}$.
        \item[1.3.]
        $\Bigl \{(V,{\mathcal{F}}_{4,14(\lambda,\mu,\varphi)}); 
        \mu,\lambda\in \mathbb{R}, \mu>0, \varphi\in(0,\pi)\Bigl \}$.
    \end{enumerate}
    We denote these types by \, ${\mathcal{F}}_1,\, {\mathcal{F}}_2,\,
    {\mathcal{F}}_3$ respectively.
    \item [2.] Furthermore, we have
    \begin{enumerate}
        \item[2.1.] The MD5-foliations of the type ${\mathcal{F}}_1$ are
        trivial fibration with connected fibres on the 3-dimensional
        unitary sphere $S^3$.
        \item[2.2.] The MD5-foliations of the types ${\mathcal{F}}_2,\,
        {\mathcal{F}}_3$ can be given by suitable actions of
        $\mathbb{R}^{2}$ on the foliated manifolds
        $V \cong{\mathbb{R} \times \left(\mathbb{R}^4\backslash \{0\} \right)}$ and their graphs are identified with $V\times \mathbb{R}^{2}$ \hfill {$\square$}
    \end{enumerate}
\end{enumerate} 
\end{prop}
As a direct consequence of Propositions 1.1.2, 1.1.3 and 1.3.3, we have the following assertion.

\vskip5mm
\begin{corol}[]{\bf (Analytical description of Connes' C*-algebras of considered MD5-foliations)} 
\begin{enumerate}
   \item[1.] The Connes' C*- algebra of all the MD5-foliations of the type ${\mathcal{F}}_1$ is isomorphic to $C(S^3)\otimes \mathcal{K}$. 
   \item[2.] The Connes' $C^*$-algebra of all the MD5-foliations of the types ${\mathcal{F}}_2, \,{\mathcal{F}}_3$ is isomorphic to the reduced crossed product 
$C_0(V)\rtimes \mathbb{R}^2.$ \hfill{$\square$}
\end{enumerate}
\end{corol}

\section{K-THEORY FOR THE LEAF SPACE OF THE MD5-FOLIATIONS OF TYPES $\mathbf{{\F}_2, \,{\F}_3}$}
\hskip6mm In this section, we study K-theory for the leaf space of all MD5-foliations of the non-trivial types $\F_2, \,\F_3$ and characterize the Connes' C*-algebras of these foliations. In the introduction, we have emphasized that the K-theory for the leaf space of all MD5-foliations of the type $\F_2$ and C*-algebra $C^*(\F_2)$ have studied in \cite{VU-HO10}. However, for the sake of completeness, we will concurrently present here results of both these types $\F_2$ and $\F_3$.  
    
In view of Proposition 1.3.3, all MD5-foliations $(V,{\mathcal{F}}_{4,12(\lambda,\varphi)})$; $\lambda, \,\lambda_1, \,\lambda_2 \in \mathbb{R} \backslash\{0\}$; $\varphi\in(0,\pi)$, are topologically equivalent to each other and they define the type ${\F}_2$. Thus, we need only to choose one envoy among them to describe the structure of the C*-algebra. Namely, we choose the foliation $\left(V,\F_{4,12\left({1,\frac{\pi}{2}} \right)}\right)$ deputising for the type ${\F}_2$. Similarly, all foliations $(V,{\mathcal{F}}_{4,14(\lambda,\mu,\varphi)})$; $\mu,\lambda\in \mathbb{R},\, \mu>0,\, \varphi\in(0,\pi)$, are topologically equivalent to each other and they define the type ${\F}_3$. We have also to choose $\left(V,{\mathcal{F}}_{4,14\left(0,1,\frac{\pi } {2}\right)}\right)$ deputising for the type ${\F}_3$. So, the Connes' C*-algebras of all foliations defining ${\F}_2$ are isomorphic to $C^*\left({V,\F_{4,12\left( {1,\frac{\pi } {2}} \right)}}\right)$ and for brevity, we denote them by $C^*({\F}_2)$.  Similarly, the Connes' C*-algebras of all foliations defining ${\F}_3$ are isomorphic to $C^*\left(V,{\mathcal{F}}_{4,14\left(0,1,\frac{\pi } {2}\right)}\right)$ and we also denote them, for brevity, by $C^*({\F}_3)$.
        
    In Proposition 1.3.3,  we have seen that MD5-foliations $\left({V,\F_{4,12\left( {1,\frac{\pi } {2}} \right)}}\right)$ and $\left(V,{\F}_{4,14\left(0,1,\frac{\pi } {2}\right)}\right)$ can be described by suitable actions of  \,$\mathbb{R}^2 $ on the foliated manifold 
$$V \cong \mathbb{R} \times \left(\mathbb{R}^4 \backslash \{0\} \right) 
\cong \mathbb{R} \times \left(\left({\mathbb{C} \times \mathbb{R}^2} \right)\backslash \{0\} \right)
\cong \mathbb{R} \times \left(\left({\mathbb{C} \times \mathbb{C}} \right)\backslash \{0\} \right).$$
Namely, it is easy to verify that MD5-foliations $\left({V,\F_{4,12\left( {1,\frac{\pi } {2}} \right)}}\right)$ and $\left(V,{\F}_{4,14\left(0,1,\frac{\pi } {2}\right)}\right)$ are given by the following actions ${\lambda}_{12}$ and ${\lambda}_{14}$ of \,$\mathbb{R}^2 $ on $V$ respectively.
\begin{equation}
   {\lambda}_{12}\bigl( {\left( {r, a} \right),\left( {x, y + iz, t, s} \right)} \bigl): = \bigl( {x + r, \left( {y + iz} \right).e^{ - ia} , t.e^a , s.e^a } \bigl)
\end{equation}
for every  $\left( {r, a} \right) \in \mathbb{R}^2 ,{\text{ }}\left( {x, y + iz, t, s} \right) \in V \cong \mathbb{R} \times \left(\left({\mathbb{C} \times \mathbb{R}^2} \right)\backslash \{0\} \right)$.  
\begin{equation}
   {\lambda}_{14}\bigl( {\left( {r, a} \right),\left( {x, y + iz, t +is} \right)} \bigl): = \bigl( {x + r, \left( {y + iz} \right).e^{ - ia} , \left( {t + is} \right).e^{ - ia}} \bigl),
\end{equation}
for every  $\left( {r, a} \right) \in \mathbb{R}^2 ,{\text{ }}\left( {x, y + iz, t + is} \right) \in V \cong \mathbb{R} \times \left(\left({\mathbb{C} \times \mathbb{C}} \right)\backslash \{0\} \right)$.

    Moreover, it follows from Proposition 1.3.3 and Corollary 1.3.4 that 
\begin{equation}
C^*({\F}_2) \cong C_0 \left( V \right)   \rtimes _{{\lambda}_{12}}  \mathbb{R}^2; 
\qquad C^*({\F}_3) \cong C_0 \left( V \right)   \rtimes _{{\lambda}_{14}}  \mathbb{R}^2.
\end{equation}    
 
\subsection{$\mathbf{C^*({\F}_2),\, C^*({\F}_3)}$ as Extensions of C*-Algebras}

\hskip6mm Many foliations are defined by decomposing suitably the given manifold, constructing foliations on the smaller pieces and then patching together to obtain a foliation of the total manifold. If the pieces when included in the total manifold are saturated with respect to the foliation, this may be viewed as patching together bits of leaf space. Translating into the language of C*-algebras, this construction of the foliation $(V, \F)$ gives us an extension of the form (1.1.1). Then, to study the K-theory of the leaf space of this foliation $(V, \F)$, we need to compute the index invariant of $C^*(V,\F)$  with respect to the constructed extension of the form (1.1.1). If considered C*-algebras are isomorphic to the reduced crossed products of the form $C_0(M)\rtimes H$ (see [11]), where $M$ is some locally compact space and $H$ is some Lie group acting on $M$, we can use the Thom-Connes isomorphism (see \cite{CO81}) to compute the connecting maps $\delta_0, \,\delta_1$.

\hskip6mm So, before we study the K-theory for the leaf spaces of MD5-foliations of the types ${\F}_2, \ \, {\F}_3$, we need to choose some suitable saturated open submanifolds in foliated manifold and construct the extensions of the form (1.1.1). 

Let  $V_1, \, W_1, \, V_2, \, W_2, \, V_3, \, W_3$ be the following submanifolds of $V$

\hskip2cm $V_1:\, = \lbrace (x, y, z, t, s) \in V \,\vert \,s \ne 0 \rbrace \cong \mathbb{R} \times \mathbb{R}^2  \times \mathbb{R} \times (\R),$

\hskip2cm $W_1:\, = V \setminus V_1 = \lbrace(x, y, z, t, s) \in V \,\vert \,s = 0 \rbrace \cong \mathbb{R} \times \left( \mathbb{R}^3 \backslash \{0\} \right)  \times \left\{ 0 \right\} \cong \mathbb{R} \times \left( \mathbb{R}^3 \backslash \{0\} \right),$

\hskip2cm $V_2:\, = \lbrace(x, y, z, t, 0) \in W_1 \,\vert \,t \ne 0 \rbrace \cong \mathbb{R} \times \mathbb{R}^2  \times (\R),$

\hskip2cm $W_2:\, = W_1 \setminus V_2 = \lbrace(x, y, z, t, 0) \in W_1 \,\vert \,t = 0 \rbrace \cong \mathbb{R} \times \left( \mathbb{R}^2 \backslash \{0\} \right),$

\hskip2cm $V_3:\, = \lbrace (x, y, z, t, s) \in V \,\vert \,t^2 + s^2 \ne 0 \rbrace \cong 
\mathbb{R} \times \mathbb{R}^2  \times \left( \mathbb{R}^2 \backslash \{0\} \right),$

\hskip2cm $W_3:\, = \, V \setminus V_3 = \, W_2$.

It is easy to see that the action  $\lambda_{12}$ (resp., $\lambda_{14}$) in Formula (2.0.1) (resp., (2.0.2)) preserves the subsets $V_1, \, W_1, \, V_2, \, W_2$ (resp., $V_3,\,  W_3$). 

Let $i_1, \, i_2, \, i_3, \, \mu_1 ,\mu_2, \mu_3$  be the inclusions and the restrictions as follows
\[ \begin{array}{*{20}c}
   {i_1 :C_0 \left( {V_1 } \right) \to C_0 \left( V \right),} \hfill & {i_2 :C_0 \left( {V_2 } \right) \to C_0 \left( {W_1 } \right),} \hfill 
& {i_3 :C_0 \left( {V_3 } \right) \to C_0 \left( V \right),}\\
   {\mu _1 :C_0 \left( V \right) \to C_0 \left( {W_1 } \right),} \hfill & {\mu _2 :C_0 \left( {W_1 } \right) \to C_0 \left( {W_2 } \right)} \hfill 
& {\mu _3 :C_0 \left( V \right) \to C_0 \left( {W_3 } \right),} \\
\end{array} \]
where each function of  $C_0 \left( {V_1 } \right)$  (resp.,  $C_0 \left( {V_2 } \right)$, \, $C_0 \left(V_3 \right)$) is extented  to the one of  $C_0 \left( V \right)$ (resp.,  $C_0 \left( {W_1 } \right)$, \,$C_0 \left( V \right)$) by taking the value of zero outside $V_1 $  (resp., $V_2 $, \, $V_3$).

    It is obvious that $i_1,\, i_2, \, \mu _1, \, \mu _2 $ (resp., $i_3, \, \mu_3$) are  $\lambda_{12}$-equivariant (resp., $\lambda_{14}$-equivariant). Moreover, the following sequences are equivariantly exact:

\begin{equation}
        \xymatrix{0 \ar[r] & C_0(V_1)\ar[r]^{i_1} & C_0(V) \ar[r]^{\mu_1} & C_0(W_1) \ar[r] & 0},
\end{equation}
\begin{equation}
        \xymatrix{0 \ar[r] & C_0(V_2) \ar[r]^{i_2} & C_0(W_1) \ar[r]^{\mu_2} & C_0(W_2) \ar[r] & 0},
\end{equation}
\begin{equation}
        \xymatrix{0 \ar[r] & C_0(V_3)\ar[r]^{i_3} & C_0(V) \ar[r]^{\mu_3} & C_0(W_3) \ar[r] & 0}.
\end{equation}

Now we denote by  $(V_1, \F_{12}), \,(W_1, \F_{12}), \,(V_2, \F_{12}), \, 
(W_2, \F_{12})$ the restrictions of \,$\left(V,\F_{4,12\left(1, \frac{\pi}{2}\right)}\right)$ to $V_1,$ \,$W_1,$ \,$V_2,$ \,$W_2$ respectively. Similarly, the restrictions of 
$\left(V,{\F}_{4,14\left(0, 1, \frac{\pi}{2}\right)}\right)$ to $V_3$ and $W_3$ are denoted by 
$(V_3, \F_{14})$ and $(W_3, \F_{14})$ respectively. 

The following theorem is the first of the main results of the paper.

\begin{thm}[]{\bf ($\mathbf{C^*({\F}_2),\, C^*({\F}_3)}$ as extensions of 
C*-algebras)}
   \begin{enumerate}
      \item[1.] $C^*({\F}_2)$ admits the following repeated extensions
        \begin{equation}\label{gamma1}
            \xymatrix{0 \ar[r] & J_1 \ar[r]^{\hspace{-.4cm}\widehat{i_1}} & C^{*}({\F}_2) \ar[r]^{\hspace{.3cm}\widehat{\mu_1}} & B_1 \ar[r] & 0}, \tag{$\gamma_1$}
        \end{equation}
        \begin{equation}\label{gamma2}
            \xymatrix{0 \ar[r] & J_2 \ar[r]^{\widehat{i_2}} & B_1 \ar[r]^{\widehat{\mu_2}} & B_2\ar[r] & 0}, \tag{$\gamma_2$}
        \end{equation}
     \item[2.] $C^*({\F}_3)$ admits the following extension
        \begin{equation}\label{gamma1}
            \xymatrix{0 \ar[r] & J_3 \ar[r]^{\hspace{-.4cm}\widehat{i_3}} & C^{*}({\F}_3) \ar[r]^{\hspace{.4cm}\widehat{\mu_3}} & B_3 \ar[r] & 0}, \tag{$\gamma_3$}
        \end{equation}
   \end{enumerate}

where 

\hskip2cm    $J_1:  = C^* \left(V_1,\F_{12} \right) \cong C_0 \left( {V_1 } \right) \rtimes _{\lambda_{12}} \mathbb{R}^2  \cong C_0 \left( \mathbb{R}^3  \sqcup \mathbb{R}^3 \right) \otimes \K,$

\hskip2cm    $J_2:  = C^* \left(V_2,\F_{12} \right) \cong C_0 \left( {V_2} \right) \rtimes _{\lambda_{12}} \mathbb{R}^2  \cong C_0 \left(\mathbb{R}^2  \sqcup \mathbb{R}^2  \right) \otimes \K,$

\hskip2cm    $B_1  = C^* \left(W_1,\F_{12} \right) \cong C_0 \left( {W_1 } \right) \rtimes _{\lambda_{12}}  \mathbb{R}^2,$
    
\hskip2cm    $B_2  = C^* \left(W_2,\F_{12} \right) \cong C_0 \left( {W_2 } \right) \rtimes _{\lambda_{12}} \mathbb{R}^2  \cong C_0 \left(\mathbb{R}_ + \right) 
\otimes \K,$
    
\hskip2cm    $J_3  = C^* \left(V_3,\F_{14} \right) \cong C_0 \left( {V_3 } \right) \rtimes _{\lambda_{14}} \mathbb{R}^2  \cong C_0 \left(\mathbb{C} \times \mathbb{R}_+ \right) \otimes \K,$
    
\hskip2cm    $B_3  = C^* \left(W_3,\F_{14} \right) \cong C_0 \left( {W_3 } \right) \rtimes _{\lambda_{14}} \mathbb{R}^2  \cong C_0 \left(\mathbb{R}_+ \right) 
\otimes \K,$\\

and the homomorphismes \, $\widehat{i_1 }, \, \widehat{i_2 }, \, \widehat{i_3 }, \, \widehat{\mu_1 }, \, \widehat{\mu_2}, \, \widehat{\mu_3}$ \, are defined by
        $$\left({\widehat{i_k }f} \right)\left( {r,s} \right) = i_k f\left({r,s} \right), \,\, \left({\widehat{\mu _k }f} \right)\left( {r,s} \right) = \mu _k f\left( {r,s} \right); \,(r, s) \in \mathbb{R}^2, \,k = 1, 2, 3.$$
\end{thm}    
\vskip1cm
\begin{prove} \end{prove}

We note that the graph of $\left(V_1,\F_{12} \right)$  is given as  $V_1  \times \mathbb{R}^2 $, so by using Proposition 1.1.2, one has 
        $$J_1  = C^* \left( V_1,\F_{12} \right) \cong C_0 \left( {V_1 } \right) \rtimes _{\lambda_{12}}  \mathbb{R}^2.$$  
Similarly, we have
\begin{description}
   \item[] \hskip2cm $B_1 \cong C_0 \left( {W_1 } \right) \rtimes _{\lambda_{12}} \mathbb{R}^2, \quad
  J_2 \cong C_0 \left( {V_2 } \right) \rtimes _{\lambda_{12}} \mathbb{R}^2, \quad
  B_2 \cong C_0 \left( {W_2 } \right) \rtimes _{\lambda_{12}} \mathbb{R}^2,$
   \item[] \hskip2cm $J_3 \cong C_0 \left( {V_3 } \right) \rtimes _{\lambda_{14}} \mathbb{R}^2, \quad
  B_3 \cong C_0 \left( {W_3 } \right) \rtimes _{\lambda_{14}} \mathbb{R}^2.$
\end{description}

Note that $V_1, \, V_3$ and $V_2$ are saturated submanifolds in $\left(V,\F_{4,12\left(1, \frac{\pi}{2}\right)}\right)$, \, $\left(V,{\F}_{4,14\left(0,1,\frac{\pi}{2}\right)}\right)$ and $(W_1, \F_{12})$ respectively. Furthermore, the graphs of all foliations $(W_1, \F_{12}), \,(W_2, \F_{12}), \, (W_3, \F_{14})$ are given as the product of foliated manifolds with $\mathbb{R}^2$. So we obtain extensions $(\gamma_1), \, (\gamma_2)$  and  
$(\gamma_3)$ by using Proposition 1.1.5. 

From the equivariantly exact sequences (2.1.1), (2.1.2), (2.1.3), it follows from (\cite[Lemma 1.1]{CO81}) that Extensions $(\gamma_1), \,(\gamma_2)$ \,(resp., $(\gamma_3)$) can be obtained from reduced crossed products of the C*-algebras in the exact sequences (2.1.1), (2.1.2) \, (resp., (2.1.3)) by $\mathbb{R}^2$.
        
On the other hand, it is easily seen that the foliation $(V_1,\F_{12})$  can be derived from the following submersion
\begin{equation}
p_1 :V_1  \cong \mathbb{R} \times \mathbb{R}^2  \times \mathbb{R} \times (\R)  \longrightarrow \mathbb{R}^3  \sqcup \mathbb{R}^3;\quad
p_1 (x, y, z, t, s):\,= \, (y, z, t, \text{sign}s).
\end{equation}
Hence, in view of Proposition 1.1.3, we get  $J_1  \cong C_0 \left( {\mathbb{R}^3  \sqcup \mathbb{R}^3 } \right) \otimes \K$. 

The same argument shows that
       
$$J_2 \, \cong \, C_0 ( \mathbb{R}^2  \sqcup \mathbb{R}^2 ) \otimes \K, \qquad
B_2 \, \cong  \, C_0 ( \mathbb{R}_+ ) \otimes \K.$$ 
$$J_3  = C^* \left(V_3,\F_{14} \right) \cong C_0 \left(\mathbb{C} \times \mathbb{R}_+ \right) \otimes \K, \qquad
  B_3  = C^* \left(W_3,\F_{14} \right) \cong C_0 \left(\mathbb{R}_+ \right) 
\otimes \K.$$ 

The proof is complete. \hfill{$\square$

\subsection{Computing the Invariant Systems of $\mathbf{C^*({\F}_2)}$ and  $\mathbf{C^*({\F}_3)}$}
\hskip6mm We now study the K-theory for the leaf spaces of MD5-foliations of the types $\F_{2}, \,\F_{3}$ and characterize $C^*(\F_{2}), \,C^*(\F_{3})$ by K-functors. Namely, we will compute the invariant systems of $C^*(\F_{2}), \,C^*(\F_{3})$ in KK-groups of Kasparov (\cite{KAS}). 

Firstly, we recall that Extensions \,$(\gamma_1), \,(\gamma_2), \,(\gamma_3)$ \, define \,$C^*(\F_{2}), \, B_1, \,C^*(\F_{3})$ \, as elements in KK-groups \,$Ext(B_1, J_1), \, Ext(B_2, J_2), \, Ext(B_3, J_3)$ \,respectively. These elements are called the invariants of \,$C^*(\F_{2}), \,C^*(\F_{3})$. Namely, we have the following definition.

\vskip6mm
\begin{definition}
The pair of elements  $\left( \gamma _1, \,\gamma _2 \right)$ \,$($resp., \,the element \, $\gamma _3)$ \,corresponding to the repeated extensions  $(\gamma_1)$, \, $(\gamma_2)$ \,$($resp., \,the extension $(\gamma_3))$ \,in $Ext(B_1, J_1)\,\oplus \,Ext(B_2, J_2)$\,$($resp., $Ext(B_3, J_3))$ \,is called the system of index invariants \,$($resp., the index invariant$)$ of \,$C^*(\F_{2})$ \,$($resp., $C^*(\F_{3}))$ \,and denoted by $Index \,C^*(\F_{2})$ \,$($resp., \,$Index \,C^*(\F_{3}))$.
\end{definition}

As the analyses in the introduction, we have the following remark.

\vskip6mm
\begin{rem}
$Index \,C^*(\F_{2}), \,Index \,C^*(\F_{3})$ determine the so-called ``stable type'' of $C^*(\F_{2}), \,C^*(\F_{3})$ in $Ext(B_1, J_1)\oplus Ext(B_2, J_2)$, 
$Ext(B_3, J_3)$, respectively.
\end{rem}

Now, we present and prove the last main result of the paper.

\vskip6mm
\begin{thm}[]{\bf (Index Invariants of $\mathbf{C^*({\F}_2)}$ and $\mathbf{C^*({\F}_3)}$)}
\begin{enumerate}
   \item[1.] $Index\,C^*(\F_{2})$ = \,$\left( \gamma_1, \,\gamma_2 \right) $,\, where
   \begin{description}
       \item[] \hskip1cm 
       $\gamma_1 \, = \left( {\begin{array}{*{20}c}
                              0 & 1  \\
                              0 & 1  \\
                             \end{array} } \right)$ 
in the KK-group $Ext(B_1, \,J_1) \, \equiv \, Hom(\mathbb{Z}^2, \, \mathbb{Z}^2)$;

      \item[] \hskip1cm
      $\gamma_2 \, = (1, 1)$ 
in the KK-group $Ext(B_2, \,J_2) \, \equiv \,  Hom(\mathbb{Z}, \, \mathbb{Z}^2)$.
   \end{description}
   \item[2.] $Index\,C^*(\F_{3})$ = \,$\gamma_3$,\, where
      $\gamma_3 \, = (0, 1)$ 
in the KK-group $Ext(B_3, \,J_3) \, \equiv \,  Hom(\mathbb{Z}, \, \mathbb{Z}) \oplus \, Hom(\mathbb{Z}, \, \mathbb{Z})$.    
\end{enumerate}
\end{thm}

To prove this theorem, we need some lemmas as follows.

\vskip6mm
\begin{lemma}[see {\bf \cite[Lemma 4.4]{VU-HO10}}]
Set $I_2 : = C_0 \left(\mathbb{R}^2  \times \left(\R\right)\right)$ and 
$A_2:  = C_0 \left( \mathbb{R}^2 \backslash \{0\} \right)$. Then, the following diagram is commutative
$$\xymatrix{\dots \ar[r] & K_j(I_2)\ar[r] \ar[d]^{\beta_1} & K_j \left( C_0 \left(\mathbb{R}^3 \backslash \{0\} \right) \right) \ar[r] \ar[d]^{\beta_1} & K_j(A_2) \ar[r] \ar[d]^{\beta_1} & K_{j+1}(I_2) \ar[r] \ar[d]^{\beta_1} & \dots\\
            \dots \ar[r] & K_{j+1}\bigl(C_0(V_2)\bigr) \ar[r] & K_{j+1}\bigl(C_0(W_1)\bigr) \ar[r] & K_{j+1}\bigl(C_0(W_2)\bigr)\ar[r] & K_j\bigl(C_0(V_2)\bigr)\ar[r] & \dots}$$
where  $\beta_1 $ is the isomorphism defined in (\cite[Theorem 9.7]{TA}) or in (\cite[Corollary VI.3]{CO81}),  $j \in \mathbb{Z}/2\mathbb{Z}$.
\end{lemma}
\begin{proof}
Let
$$k_2 :I_2  = C_0 \left( {\mathbb{R}^2  \times \left(\R\right) } \right) \xrightarrow{{}} C_0 \left( \mathbb{R}^3 \backslash \{0\} \right),
     \qquad v_2 : C_0 \left( \mathbb{R}^3 \backslash \{0\} \right) \xrightarrow{{}}A_2  = C_0 \left( \mathbb{R}^2 \backslash \{0\} \right)$$
be the inclusion and restriction defined similarly as $i_1, \, i_2, \, i_3, \, \mu_1 ,\mu_2, \mu_3$ in Subsection 2.1. One gets the exact sequence as follows

\begin{equation}
\xymatrix{0\ar[r] & I_2 \ar[r]^{\hspace{-.7cm}k_2} & C_0 \left( \mathbb{R}^3 \backslash \{0\} \right) \ar[r]^{\hspace{.7cm}v_2} & A_2 \ar[r] & 0}.
\end{equation}

Note that
\begin{description}
   \item \hskip3cm $C_0(V_2) \cong C_0 \left(\mathbb{R} \times \mathbb{R}^2  \times (\R) \right) \cong C_0 \left( \mathbb{R} \right) \otimes I_2,$

   \item \hskip3cm $C_0(W_2) \cong C_0 \left(\mathbb{R} \times \left(\mathbb{R}^2 \backslash \{0\} \right) \right) \cong C_0 \left(\mathbb{R} \right) \otimes A_2,$

   \item \hskip3cm $C_0(W_1) \cong C_0 \left(\mathbb{R} \times \left( \mathbb{R}^3 \backslash \{0\} \right) \right) \cong C_0 \left(\mathbb{R} \right) \otimes C_0 \left( \mathbb{R}^3 \backslash \{0\} \right).$
\end{description}
        
So the extension (2.1.2) can be identified to the following one
\begin{equation}
 \xymatrix{0\ar[r] & C_0(\mathbb{R})\otimes I_2 \ar[r]^{\hspace{-.8cm}id\otimes k_2} & C_0(\mathbb{R})\otimes C_0 \left( \mathbb{R}^3 \backslash \{0\} \right) \ar[r]^{\hspace{.6cm}id\otimes v_2} & C_0(\mathbb{R})\otimes A_2\ar[r] & 0}.
\end{equation}

Now, using \cite[Theorem 9.7, Corollary 9.8]{TA}, we obtain the assertion of the lemma.
\end{proof}

By an argument analogous to that used for the above proof, by using the exact sequence (2.1.1) and diffeomorphisms
$V \cong \mathbb{R} \times \left(\mathbb{R}^4 \backslash \{0\} \right) \cong \mathbb{R} \times \mathbb{R}_+ \times S^3$,\,
$W_1 \cong \mathbb{R} \times \left(\mathbb{R}^3 \backslash \{0\} \right) \cong \mathbb{R} \times \mathbb{R}_+ \times S^2 $ we obtain the following lemma.

\begin{lemma}[see {\bf \cite[Lemma 4.5]{VU-HO10}}]
        Set $I_1  = C_0 \left( {\mathbb{R}^2  \times (\R) } \right)$ and $A_1  = C\left( {S^2 } \right)$. Then, the following diagram is commutative
$$\xymatrix{\dots\ar[r] & K_j(I_1) \ar[r]\ar[d]^{\beta_2} & K_j\bigl(C(S^3)\bigr)\ar[r]\ar[d]^{\beta_2} & K_j(A_1)\ar[r]\ar[d]^{\beta_2} & K_{j+1}(I_1)\ar[r]\ar[d]^{\beta_2} & \dots\\
            \dots \ar[r] & K_j(C_0(V_1))\ar[r] & K_j(C_0(V)) \ar[r] & K_j(C_0(W_1))\ar[r] & K_{j+1}(C_0(V_1))\ar[r] & \dots}$$
where  $\beta _2 $ is the Bott isomorphism,  $j \in \mathbb{Z}/2\mathbb{Z}$.
\hfill {$\square$}
\end{lemma}

Similarly, using the exact sequence (2.1.3) and diffeomorphisms
$V \cong \mathbb{R} \times \left(\mathbb{R}^4 \backslash \{0\} \right) \cong \mathbb{R} \times \mathbb{R}_+ \times S^3$, $W_3 \cong \mathbb{R} \times \left(\mathbb{R}^2 \backslash \{0\} \right) \cong \mathbb{R} \times \mathbb{R}_+ \times S^1,$ $V_3 \cong \mathbb{R} \times \mathbb{R}^2 \times \left(\mathbb{R}^2 \backslash \{0\} \right) \cong \mathbb{R} \times \mathbb{R}^2 \times \mathbb{R}_+ \times S^1$, \,we also get the following lemma.

\vskip6mm
\begin{lemma}
        Set $I_3  = C_0(\mathbb{R}^2 \times S^1)$ and $A_3  = C(S^1)$. Then, the following diagram is commutative
$$\xymatrix{\dots\ar[r] & K_j(I_3) \ar[r]\ar[d]^{\beta_2} & K_j\bigl(C(S^3)\bigr)\ar[r]\ar[d]^{\beta_2} & K_j(A_3)\ar[r]\ar[d]^{\beta_2} & K_{j+1}(I_3)\ar[r]\ar[d]^{\beta_2} & \dots\\
            \dots \ar[r] & K_j(C_0(V_3))\ar[r] & K_j(C_0(V)) \ar[r] & K_j(C_0(W_3))\ar[r] & K_{j+1}(C_0(V_3))\ar[r] & \dots}$$
where  $\beta _2 $ is the Bott isomorphism,  $j \in \mathbb{Z}/2\mathbb{Z}$.
\hfill {$\square$}
\end{lemma}

Before computing the K-groups, we need the following notations. Let $u_\pm:\mathbb{R_\pm} \to S^1 $ be the map
$$u_\pm(z): = e^{2\pi i (\mp z/\sqrt {1 + z^2})}, \, z \in \mathbb{R_\pm}.$$
Note that the class $[u_+]$ 
(resp.,  $[u_-]$) is the generator of  $K_1 \left(C_0(\mathbb{R}_+) \right) \cong \mathbb{Z}$ (resp.,  $K_1 \left(C_0(\mathbb{R}_-) \right) \cong \mathbb{Z}$). Let us consider the matrix valued function  $p: \, \mathbb{R}^2 \backslash \{0\}  \cong S^1 \times \mathbb{R}_+ \longrightarrow M_2(\mathbb{C})$ (resp.,  $\widehat{p}: \, S^2  \cong D^2/S^1 \to M_2(\mathbb{C})$) defined by:

 $$p(x, y)\left(resp., \,\widehat{p}(x, y) \right): = \frac{1}
{2}\left( {\begin{array}{*{20}c}
   {1 - \cos \pi \sqrt {x^2  + y^2 } } & {\frac{{x + iy}}
{{\sqrt {x^2  + y^2 } }}\sin \pi \sqrt {x^2  + y^2 } }  \\
   {\frac{{x - iy}}
{{\sqrt {x^2  + y^2 } }}\sin \pi \sqrt {x^2  + y^2 } } & {1 + \cos \pi \sqrt {x^2  + y^2 } }  \\
 \end{array} } \right).$$

\vskip6mm
    Then $p$ (resp., $\widehat{p}$) is an idempotent of rank 1 for each  $(x, y) \in \mathbb{R}^2 \backslash \{0\}$ (resp.,  $(x, y) \in D^2/S^1 $). Let $[b] \in K_0 \left(C_0(\mathbb{R}^2) \right)$  be the Bott element, and [1] be the generator of  $K_0 \left(C(S^1) \right) \cong \mathbb{Z}$.

\vskip6mm
\begin{lemma}[]{\bf (see \cite[Lemma 4.6]{VU-HO10})}
\begin{enumerate}
   \item[(i)] $K_0(B_1)\cong \mathbb{Z}^2,\ K_1(B_1)=0$,
   \item[(ii)] $K_0(J_2)\cong \mathbb{Z}^2$ is generated by $\varphi_0\beta_1\bigl([b]\boxtimes[u_+]\bigr)$ and $\varphi_0\beta_1\bigl([b]\boxtimes[u_-]\bigr);\ K_1(J_2)=0$,
   \item[(iii)] $K_0(B_2)\cong \mathbb{Z}$ is generated by $\varphi_0\beta_1\bigl([1]\boxtimes[u_+]\bigr);\ K_1(B_2)\cong \mathbb{Z}$ is generated by $\varphi_1\beta_1\bigl([p]-[\varepsilon_1]\bigr)$,
\end{enumerate}
where $\varphi_j$ is the Thom-Connes isomorphism, \, $j\in \lbrace 0, 1 \rbrace$ (see[3]); $\beta_1$ is the isomorphism in Lemma 2.2.4; $\varepsilon_1$ is the constant matrix $\biggl(\begin{matrix}
            1 & 0\\
            0 & 0
        \end{matrix}\biggr)$
and $\boxtimes$ is the external tensor product (see, for example, [4,VI.2]).
\hfill{$\square$}
\end{lemma}

\begin{lemma}[]{\bf (see \cite [Lemma 4.7]{VU-HO10})}
\begin{enumerate}
    \item[(i)] $K_0\bigl(C^{*}(\F_2)\bigr)\cong\mathbb{Z},\ K_1\bigl(C^*(\F_2)\bigr)\cong\mathbb{Z}$,
    \item[(ii)] $K_0(J_1)=0;\ K_1(J_1)\cong \mathbb{Z}^2$ is generated by $\varphi_1\beta_2\bigl([b]\boxtimes[u_+]\bigr)$ and $\varphi_1\beta_2\bigl([b]\boxtimes[u_-]\bigr)$,
    \item[(iii)] $K_1(B_1)=0;\ K_0(B_1)\cong \mathbb{Z}^2$ is generated by $\varphi_0\beta_2[\widehat{1}]$ and $\varphi_0\beta_2\bigl([\widehat{p}]-[\varepsilon_1]\bigr)$,
\end{enumerate}
where $\widehat{1}$ is unit element in $C(S^2)$; $\varphi_j$ is the Thom-Connes isomorphism, \,$j\in \lbrace 0, 1 \rbrace$; $\beta_2$ is the Bott isomorphism.
\hfill{$\square$}
\end{lemma}

\begin{rem} It is easily seen that
\begin{enumerate}
    \item[(i)] $K_i\bigl(C_0(\mathbb{R}^2 \times S^1)\bigr) \, \cong \, 
    K_i\bigl(C(S^1)\bigr)\, \cong \, \mathbb{Z}; \, i = 0, 1,$
    \item[(ii)] $K_i\bigl((S^3) \bigr)\,\cong \,\mathbb{Z}; i = 0, 1,$
    \item[(iii)] $K_0\bigl(C(S^1)\bigr)\,\cong\,\mathbb{Z}$ is generated by $\varphi_0\beta_2[1]$ and $K_1\bigl(C(S^1)\bigr)\,\cong\,\mathbb{Z}$ is generated by $\varphi_1\beta_2[Id]$,
\end{enumerate}
where $\varphi_0, \varphi_1$ are the Thom-Connes isomorphisms, $\beta_2$ is the Bott isomorphism.
\end{rem}

We are now prove Theorem 2.2.3.

\begin{proof}[see {\bf Proof of Theorem 2.2.3}]
\renewcommand{\qedsymbol}{}
\end{proof}

{\bf 1. Computation of \, $\mathbf{\gamma_1}$ \,(see \cite[pp. 257--258]{VU-HO10})}

Recall that the extension $(\gamma_1)$ in Theorem 2.1.1 gives rise to a six-term exact sequence
        $$\xymatrix{0 = K_0(J_1)\ar[r] & K_0\bigl(C^*(\F_2)\bigr)\ar[r] & K_0(B_1)\ar[d]^{\delta_0}\\
        0=K_1(B_1)\ar[u]^{\delta_1} & K_1\bigl(C^*(\F_2)\bigr) \ar[l] & K_1(J_1)\ar[l]}$$
        
By \cite[Theorem 4.14]{RO82}, the isomorphism
        $$\Ext(B_1,J_1)\xrightarrow{\cong} \Hom\bigl((K_0(B_1),K_1(J_1)\bigr) \equiv \Hom(\mathbb{Z}^2,\mathbb{Z}^2)$$ 
associates the invariant $\gamma_1\in \Ext(B_1,J_1)$ to the connecting map $\delta_0:K_0(B_1)\rightarrow K_1(J_1)$. Since the Thom-Connes isomorphism commutes with $K-$theoretical exact sequence (\cite[Lemma 3.4.3]{TO}), we have the following commutative diagram $(j\in \mathbb{Z}/2\mathbb{Z})$:
        $$\xymatrix{\dots \ar[r] & K_j(J_1)\ar[r] & K_j\bigl(C^*(\F_2)\bigr)\ar[r] & K_j(B_1)\ar[r]& K_{j+1}(J_1)\ar[r] & \dots\\
        \dots \ar[r] & K_j\bigl(C_0(V_1)\bigr)\ar[r]\ar[u]_{\varphi_j} & K_j\bigl(C_0(V)\bigr)\ar[r]\ar[u]_{\varphi_j} & K_j\bigl(C_0(W_1)\bigr) \ar[r]\ar[u]_{\varphi_j} & K_{j+1}\bigl(C_0(V_1)\bigr)\ar[r]\ar[u]_{\varphi_{j+1}}& \dots}$$
        
In view of Lemma 2.2.5, the following diagram is commutative
        $$\xymatrix{\dots \ar[r] & K_j\bigl(C_0(V_1)\bigr)\ar[r] & K_j\bigl(C_0(V)\bigr)\ar[r] & K_j\bigl(C_0(W_1)\bigr)\ar[r] & K_{j+1}\bigl(C_1(V_1)\bigr)\ar[r] & \dots\\
        \dots \ar[r] & K_j(I_1)\ar[r]\ar[u]_{\beta_2} & K_j\bigl(C(S^3)\bigr)\ar[r]\ar[u]_{\beta_2} & K_j(A_1)\ar[r]\ar[u]_{\beta_2} & K_{j+1}(I_1)\ar[r]\ar[u]_{\beta_2} & \dots}$$
        
Consequently, instead of computing $\delta_0: K_0(B_1)\rightarrow K_1(J_1)$, it is sufficient to compute $\delta_0: K_0(A_1)\rightarrow K_1(I_1)$. Namely, we have to calculate $\delta_0\bigl([\widehat{p}]-[\varepsilon_1]\bigr)=\delta_0\bigl([\widehat{p}]\bigr)$ because $\delta_0\bigl([\varepsilon_1]\bigr)=(0;0)$ and $\delta_0\bigl([\widehat{1}]\bigr)=(0;0)$. By the usual definition (\cite[p.170]{TA}), for $[\widehat{p}]\in K_0(A_1)$, we have  $\delta_0\big([\widehat{p}]\bigr)\, = \,\bigl[e^{2\pi i\widetilde{p}}\bigr]\in K_1(I_1)$ where $\widetilde{p}$ is a preimage of $\widehat{p}$ in (the matrix algebra over) $C(S^3)$, i.e. $v_1 \bigl(\widetilde{p} \bigr) = \widehat{p}$. Concretely, we can choose $\widetilde{p}(x, y, z):\, =\, \dfrac{1}{\sqrt{1+z^2}}.\,\widehat{p}(x, y)\, for\,(x, y, z)\in S^3$. Let $\widetilde{p}_+$ (resp. $\widetilde{p}_-$) be the restriction of $\widetilde{p}$ on $\mathbb{R}^2\times \mathbb{R}_+$ (resp. $\mathbb{R}^2\times\mathbb{R}_-$). Then we have
\begin{multline}
\delta_0\bigl([\widehat{p}]\bigr)\, = \,\bigl[e^{2\pi i\widetilde{p}}\bigr] \, = \,\bigl[e^{2\pi i\widetilde{p}_+}\bigr] + \bigl[e^{2\pi i\tilde{p}_-}\bigr]\\
\in K_1\Bigl(C_0\bigl(\mathbb{R}^2\bigr)\otimes C_0\bigl(\mathbb{R}_+\bigr)\Bigr)\oplus K_1\Bigl(C_0\bigl(\mathbb{R}^2\bigr)\otimes C_0\bigl(\mathbb{R}_-\bigr)\Bigr)\, = \, K_1(I_1).
\end{multline}
        
By \cite[Section 4]{TA}, for each function $f:\mathbb{R}_{\pm}\rightarrow Q_n\widetilde{C_0\bigl(\mathbb{R}^2\bigr)}$ such that $\displaystyle \lim_{x\rightarrow 0}f(t)=\lim_{x\rightarrow \pm\infty}f(t)$, where $Q_n\widetilde{C_0\bigl(\mathbb{R}^2\bigr)}=\Bigl\{a \in M_n\widetilde{C_0\bigl(\mathbb{R}^2\bigr)}, e^{2\pi ia}=Id\Bigr\}$, the class $[f]\in K_1\bigl(C_0(\mathbb{R}^2)\otimes C_0(\mathbb{R}_\pm)\bigr)$ can be determined by $[f]=W_f.[b]\boxtimes[u_\pm]$, where $\displaystyle W_f=\dfrac{\mp1}{2\pi i}\int_{\mathbb{R}_\pm} Tr\bigl(f'(z)f^{-1}(z)\bigr)dz$\, is the winding number of $f$.

By simple computation, we get $\delta_0\bigl([p]\bigr)=[b]\boxtimes[u_+]+[b]\boxtimes [u_-]$. Thus 
$\gamma_1 \, = \,
\left( {\begin{array}{*{20}c}
                              0 & 1  \\
                              0 & 1  \\
                             \end{array} } \right)$ \, 
$\in \Hom_{\mathbb{Z}}(\mathbb{Z}^2,\mathbb{Z}^2)$.

{\bf 2. Computation of \, $\mathbf{\gamma_2}$ \,(\cite[p.258]{VU-HO10})}

The extension $(\gamma_2)$ gives rise to a six-term exact sequence
        $$\xymatrix{K_0(J_2)\ar[r] & K_0(B_1) \ar[r] & K_0(B_2)\ar[d]^{\delta_0}\\
        K_1(B_2)\ar[u]_{\delta_1} & K_1(B_1)\ar[l] & K_1(J_2)=0\ar[l]}$$
        
By \cite[Theorem 4.14]{RO82}, $\gamma_2 = \delta_1 \in \Hom\bigl(K_1(B_2),K_0(J_2)\bigr)=\Hom_{\mathbb{Z}}(\mathbb{Z},\mathbb{Z}^2)$. Similarly to the computation of $\gamma_1$, taking account of Lemmas 2.2.4 and 2.2.7, we have the following commutative diagram 
        $$\xymatrix{\dots \ar[r] & K_j(J_2)\ar[r] & K_j(B_1)\ar[r] & K_j(B_2)\ar[r] & K_{j+1}(J_2)\ar[r] & \dots\\
        \dots \ar[r] & K_j\bigl(C_0(V_2)\bigr)\ar[r] \ar[u]^{\varphi_j}& K_j\bigl(C_0(W_1)\bigr)\ar[r]\ar[u]^{\varphi_j} & K_j\bigl(C_0(W_2)\bigr)\ar[r]\ar[u]^{\varphi_j} & K_{j+1}\bigl(C_0(V_2)\bigr)\ar[r]\ar[u]^{\varphi_{j+1}} & \dots\\
        \dots \ar[r] & K_{j-1}(I_2)\ar[r]\ar[u]^{\beta_1} & K_{j-1}\bigl(C_0(\mathbb{R}^3\backslash \{0\})\bigr) \ar[r]\ar[u]^{\beta_1} & K_{j-1}(A_2)\ar[r]\ar[u]^{\beta_1} & K_j(I_2)\ar[r]\ar[u]^{\beta_1}& \dots}$$
where $j\in \mathbb{Z}/2\mathbb{Z}$.
        
Thus we can compute $\delta_0: K_0(A_2)\rightarrow K_1(I_2)$ instead of $\delta_1: K_1(B_2)\rightarrow K_0(J_2)$. By the proof of Lemma 2.2.7, we have to define $\delta_0\bigl([p]-[\epsilon_1]\bigr)=\delta_0\bigl([p]\bigr)$ (because $\delta_0\bigl([\epsilon_1]\bigr)=(0,0)$). The same argument as above, we get $\delta_0\bigl([p]\bigr)=[b]\boxtimes[u_+]+[b]\boxtimes[u_-]$. Thus $\gamma_2=(1,1)\in \Hom_{\mathbb{Z}}(\mathbb{Z},\mathbb{Z}^2)\cong \mathbb{Z}^2$. 

{\bf 3. Computation of \, $\mathbf{\gamma_3}$}

The extension $(\gamma_3)$ gives rise to a six-term exact sequence
\begin{equation}
\xymatrix{K_0(J_3)\ar[r] & K_0 \left(C^*(\F_3)\right) \ar[r] & K_0(B_3)\ar[d]^{\delta_0}\\
        K_1(B_3)\ar[u]^{\delta_1} & K_1 \left(C^*(\F_3)\right) \ar[l] & K_1(J_3)\ar[l]}
\end{equation}
                
By \cite[Theorem 4.14]{RO82}, we have 
$$\gamma_3 \, = \, (\delta_0, \,\delta_1) \, \in Ext (B_3, \, J_3) \cong \Hom_{\mathbb{Z}} \bigl(K_0(B_3),K_1(J_3)\bigr)\,\oplus \,\Hom_{\mathbb{Z}}\bigl(K_1(B_3), \,K_0(J_3)\bigr).$$ 

Similarly to the computation of $\gamma_1$, taking account of Lemma 2.2.6 and Remark 2.2.9, for $j\in \mathbb{Z}/2\mathbb{Z}$, we have the following commutative diagram 
        $$\xymatrix{\dots \ar[r] & K_j(J_3)\ar[r] & K_j \left(C^*(\F_3)\right)\ar[r] & K_j(B_3)\ar[r] & K_{j+1}(J_3)\ar[r] & \dots\\
        \dots \ar[r] & K_j\bigl(C_0(V_3)\bigr)\ar[r] \ar[u]^{\varphi_j}& K_j\bigl(C_0(V)\bigr)\ar[r]\ar[u]^{\varphi_j} & K_j\bigl(C_0(W_3)\bigr)\ar[r]\ar[u]^{\varphi_j} & K_{j+1}\bigl(C_0(V_3)\bigr)\ar[r]\ar[u]^{\varphi_{j+1}} & \dots\\
        \dots \ar[r] & K_{j}(I_3)\ar[r]\ar[u]^{\beta_2} & K_{j}\bigl(C({S}^3)\bigr) \ar[r]\ar[u]^{\beta_2} & K_{j}(A_3)\ar[r]\ar[u]^{\beta_2} & K_{j + 1}(I_3)\ar[r]\ar[u]^{\beta_2}& \dots}$$
        
Thus we can compute the pair $(\delta_0, \delta_1) \in 
Hom_{\mathbb{Z}}\bigl(K_0(A_3), K_1(I_3)\bigr) \oplus 
Hom_{\mathbb{Z}}\bigl(K_1(A_3), K_0(I_3)\bigr)$ instead of the pair $(\delta_0, \delta_1) \in Hom_{\mathbb{Z}}\bigl(K_0(B_3), K_1(J_3)\bigr) \oplus 
Hom_{\mathbb{Z}}\bigl(K_1(B_3), K_0(J_3)\bigr)$. By Remark 2.2.9, the K-theory exact sequence (2.2.4) becomes
\begin{equation}
\xymatrix{\mathbb{Z}\ar[r] & \mathbb{Z} \ar[r] & \mathbb{Z}\ar[d]^{\delta_0}\\
        \mathbb{Z}\ar[u]^{\delta_1} & \mathbb{Z} \ar[l] & \mathbb{Z}\ar[l]}
\end{equation}
Moreover, it follows from the exactness of (2.2.5) that this exact sequence is one of the following ones
\begin{equation}
\xymatrix{\mathbb{Z}\ar[r]^0 & \mathbb{Z} \ar[r]^1 & \mathbb{Z}\ar[d]^{\delta_0 = 0}\\
        \mathbb{Z}\ar[u]^{\delta_1 = 1} & \mathbb{Z} \ar[l]^0 & \mathbb{Z}\ar[l]^1}
\end{equation}
or
\begin{equation}
\xymatrix{\mathbb{Z}\ar[r]^1 & \mathbb{Z} \ar[r]^0 & \mathbb{Z}\ar[d]^{\delta_0 = 1}\\
        \mathbb{Z}\ar[u]^{\delta_1 = 0} & \mathbb{Z} \ar[l]^1 & \mathbb{Z}\ar[l]^0}
\end{equation} 

Now, in order to complete the computation of \,$\gamma_3$, we need consider homomorphism $\delta_1$. Namely, we consider $a = e^{i\varphi} \in GL_1 \left(C(S^1)\right)$, set $b:\, = a^{-1}$. Then, one has
$$ a\oplus b = 
\left( {\begin{array}{*{20}c}
                              e^{i\varphi} & 0  \\
                              0 & e^{- i\varphi}  \\
                             \end{array} } \right) \, 
\in GL_2^0\left( C(S^1) \right).$$ 

Let $u$ be the function of matrix values defined as follows
$$u(\theta_1, \theta_2, \varphi):\, =
\left( {\begin{array}{*{20}c}
e^{i\varphi}.e^{i\theta_1}.\cos\theta_2 & - \sin\theta_2 \\
\sin\theta_2 & e^{- i\varphi}.e^{- i\theta_1}.\cos\theta_2 \\
                             \end{array} } \right) \, 
\in GL_2^0\left( C(S^3) \right)$$ 
for every $(\theta_1, \theta_2, \varphi) \equiv 
(\cos{\theta_1}.\cos{\theta_2}\cos{\varphi}, \,\cos{\theta_1}.\cos{\theta_2}\sin{\varphi}, \,\cos{\theta_1}.\sin{\theta_2}, \,\sin{\theta_1}) \in S^3$.
It is easily seen that $u$ is one preimage of $a\oplus b$ in $GL_2^0\left( C(S^1) \right)$.
Set $q:\, = I_1 \oplus O_1 = 
\left( {\begin{array}{*{20}c}
                              1 & 0  \\
                              0 & 0  \\
                             \end{array} } \right)$ \, and \,$p:\, = uqu^{-1}$,
where $I_n, \,O_n$ are the identity and zero in $Mat_n(\mathbb{R})$, respectively 
($0 < n \in \mathbb{N})$.
Then $p \in P_2 {\left(C_0(\mathbb{R}^2 \times S^1)\right)}^{+}$ \,and \,$rank(p) = 1$. By the usual definition of $\delta_1$ (see \cite[Chapter 9]{RO-LA-LA}), we get 
$\delta_1([a]) = [p] - [I_1] \ne 0 \in K_0 \left(C_0(\mathbb{R}^2 \times S^1)\right)$. So the K-theory exact sequence (2.2.4) becomes 
\begin{equation}
\xymatrix{\mathbb{Z}\ar[r]^0 & \mathbb{Z} \ar[r]^1 & \mathbb{Z}\ar[d]^{\delta_0 = 0}\\
        \mathbb{Z}\ar[u]^{\delta_1 = 1} & \mathbb{Z} \ar[l]^0 & \mathbb{Z}\ar[l]^1}
\end{equation}
The proof is complete. \hfill {$\square$}

\subsection*{Acknowledgements} The authors would like take this opportunity to thank Prof. DSc. Do Ngoc Diep for his excellent advice and financial support.  They also wish to thank the University of Economics and Law, Vietnam National University-Ho Chi Minh City and the University of Education at Ho Chi Minh City for financial supports.

\end{document}